\documentclass[preprint,12pt]{elsarticle}
\usepackage{graphicx}
\usepackage[cp1251]{inputenc}
\usepackage{amssymb}
\usepackage{amsthm}
\usepackage{cmap}

\usepackage{amsmath}% http://ctan.org/pkg/amsmath

\biboptions{sort&compress}

\newcommand{\sysn}{\left\{\begin{array}{rcl}}
\newcommand{\sysk}{\end{array}\right.}

\newtheorem{theorem}{Theorem}[section]
\newtheorem{lemma}[theorem]{Lemma}

\theoremstyle{example}

\newtheorem{proposition}[theorem]{Proposition}
\theoremstyle{definition}
\newtheorem{definition}[theorem]{Definition}
\newtheorem{corollary}[theorem]{Corollary}

\journal{...}

\begin{document}

\title{On Baire property, compactness and completeness properties of spaces of Baire functions}

\author{Alexander V. Osipov}

\address{Krasovskii Institute of Mathematics and Mechanics, \\ Ural Federal
 University, Yekaterinburg, Russia}

\ead{OAB@list.ru}

\begin{abstract} In this paper,  we have obtained  a characterization
topological spaces $X$ for which the function space
$B_{\alpha}(X,Y)$ of Baire functions of class $\alpha$  is Baire for every $1\leq\alpha\leq \omega_1$ and any
Fr\'{e}chet space $Y$. In particular, we proved that
$B_{\alpha}(X,\mathbb{R})$ is Baire if and only if
$B_{\alpha}(X,Y)$ is Baire  for any Banach space $Y$.

Also we proved that many completeness and compactness properties
coincide in spaces $B_{\alpha}(X,Y)$ for any Fr\'{e}chet space $Y$.
\end{abstract}
%\tnotetext[label1]{The research has been supported by .}


\maketitle %%
\section{Introduction}

For topological spaces $X$ and $Y$, we denote by $C_p(X,Y)$ the
set $C(X,Y)$ of all continuous functions from $X$ to $Y$ endowed
with the topology of pointwise convergence. A function
$f:X\rightarrow Y$ from a topological space $X$ to a topological
space $Y$ is a {\it Baire-one function} (or a {\it function of the
first Baire class}), if $f$ is a pointwise limit of a sequence of
 $\{f_n: n\in \mathbb{N}\}\subset C(X,Y)$. Let $B_1(X,Y)$ denote the family of all
Baire-one functions on a topological space $X$ to a topological
space $Y$ endowed with the topology of pointwise convergence. For
every countable ordinal $\alpha$, let $B_{\alpha}(X,Y)$ be a set
of all functions $f:X\rightarrow Y$ that are pointwise limits of
sequences $\{f_n: n\in \mathbb{N}\}\subset C_p(X,Y)\cup
\bigcup\limits_{\beta<\alpha} B_{\beta}(X,Y)$. For
$\alpha=\omega_1$, let
$B_{\omega_1}(X,Y)=\bigcup\limits_{\beta<\omega_1} B_{\beta}(X,Y)$
be a space of all Baire functions on a topological space $X$ to a
topological space $Y$ endowed with the topology of pointwise
convergence.

Throughout this paper,  {\it all spaces are assumed to be
Tychonoff}. This is due to the fact that for any topological space
$X$ there exists the Tychonoff space $\tau X$ and the onto map
$\tau_X: X\rightarrow \tau X$ such that for any map $f$ of $X$ to
a Tychonoff space $Y$ there exists a map $g: \tau X\rightarrow Y$
such that $f=g\circ \tau_X$. Evidently, for any map $f:
X\rightarrow Y$ the unique map $\tau f: \tau X\rightarrow \tau Y$
is defined such that $\tau f \circ \tau_X=\tau_Y\circ f$. The
correspondence $X\rightarrow \tau X$ for any space $X$ and
$f\rightarrow \tau f$ for any map $f$ is called the Tychonoff
functor \protect\cite{Is,Ok}.

Note that for any Tychonoff space $X$ and a path-connected
topological space $Y$ the space $B_{\alpha}(X,Y)$ is dense in the
Tychonoff product $Y^X$ (see Proposition 2.4 in \cite{BG} for
$C_p(X,Y)$).

\section{Baire property of $B_{\alpha}(X,Y)$}

 A space is {\it meager} (or {\it of the first Baire category}) if it
can be written as a countable union of closed sets with empty
interior. A topological space $X$ is {\it Baire} if the Baire
Category Theorem holds for $X$, i.e., the intersection of any
sequence of open dense subsets of $X$ is dense in $X$. Clearly, if
$X$ is Baire, then $X$ is not meager. The reverse implication is
in general not true. However, it holds for every homogeneous space
$X$ (see Theorem 2.3 in \protect\cite{LM}). Since the space
$B_{\alpha}(X,Y)$ is a topological group ( and, hence, is a
homogeneous space)  for any topological group $Y$,
$B_{\alpha}(X,Y)$ is Baire if and only if it is not meager.

 Being a Baire space is an
important topological property for a space and it is therefore
natural to ask when function spaces are Baire. In general, it is
not an easy task to characterize when a function space has the
Baire property. The problem for $C_p(X,\mathbb{R})$ was solved
independently by Pytkeev \protect\cite{pyt1}, Tkachuk
\protect\cite{tk} and van Douwen \protect\cite{vD}.

One of the interesting problems for the space of Baire functions
is the Banakh-Gabriyelyan problem (Problem 1.1 in
\protect\cite{BG}): {\it Let $\alpha$ be a countable ordinal.
Characterize topological spaces $X$ and $Y$ for which the function
space $B_{\alpha}(X,Y)$ is Baire}.

\medskip

Recall that a topological vector space $X$ is a {\it Fr\'{e}chet
space} if $X$ is a locally convex complete metrizable space.

In \cite{ospyt1} (Theorem 3.15), it is proved that
$C_p(X,\mathbb{R})$ is Baire if and only if $C_p(X,Y)$ is Baire
for any Fr\'{e}chet spaces $Y$.

In \cite{BG} (see Corollary 4.2), it is proved that for every
countable ordinal $\alpha\geq 2$, the function space
$B_{\alpha}(X,\mathbb{R})$ is Baire.

In \cite{Os}, we have obtained a characterization for topological
space $X$ for which the function space $B_1(X,\mathbb{R})$ is
Baire.

In this paper, we get answer to the Banakh-Gabriyelyan problem for
any topological space $X$, every $0\leq\alpha\leq \omega_1$ and
for any Fr\'{e}chet space $Y$.

\medskip

Recall that a mapping $\varphi: K\rightarrow M$ is called {\it
almost open} if $Int \varphi(V)\neq \emptyset$ for any non-empty
open subset $V$ of $K$. Note that irreducible mappings defined on
compact spaces are almost open. Also note that if
$\varphi_{\alpha}:K_{\alpha}\rightarrow M_{\alpha}$ ($\alpha\in
A$) are surjective almost open mappings then the product mapping
$\prod\limits_{\alpha\in A}\varphi_{\alpha}:
\prod\limits_{\alpha\in A} K_{\alpha}\rightarrow
\prod\limits_{\alpha\in A} M_{\alpha}$ is also almost open.

\begin{lemma}(Lemma 3.15 in \cite{ospyt})\label{lem5} Let $\psi:P\rightarrow L$ be a surjective continuous almost open
mapping and let $E\subseteq P$ be a dense non-meager (Baire)
subspace in $P$. Then $\psi(E)$ is non-meager (Baire).
\end{lemma}

\begin{lemma}(see Lemma 3.14 in \cite{ospyt} for $C_p(X,Y)$)\label{lem6} Let $\psi: K\rightarrow M$ be a surjective continuous almost open
mapping, $0<\alpha\leq \omega_1$, and let $B_{\alpha}(X,K)$ be a
non-meager (Baire) dense subspace of $K^X$. Then $B_{\alpha}(X,M)$
is non-meager (Baire).
\end{lemma}

\begin{proof} Define the mapping $\psi^X: K^X\rightarrow M^X$ such
that $\psi^X(f)=\psi\circ f$ for every $f\in K^X$.  The mapping
$\psi^X$ is almost open and continuous. Since $B_{\alpha}(X,K)$ is
dense in $K^X$ and is non-meager (Baire), by Lemma \ref{lem5},
$\psi^X(B_{\alpha}(X,K))$ is non-meager (Baire). Since $\psi^X$ is
continuous, $\psi^X(\lim f_n)=\lim \psi^X(f_n)$ for any sequence
$\{f_n: n\in \mathbb{N}\}\subset B_{\beta}(X,K)$ ($\beta<\alpha$)
and, hence, $\psi^X(B_{\alpha}(X,K))=\{\psi\circ f: f\in
B_{\alpha}(X,K)\}\subseteq B_{\alpha}(X,M)$. Thus,
$B_{\alpha}(X,M)$ contains a dense non-meager (Baire) subspace,
hence, it is also non-meager (Baire).
\end{proof}

It is well-known that there are Baire spaces $X$ and $Y$ such that
$X\times Y$ is not Baire \protect\cite{Ox1}. For the product
$\prod\limits_{\beta\in A} B_{\alpha}(X_{\beta},\mathbb{R})$ we
have the following result.

\begin{lemma}\label{lem1} If $B_{\alpha}(X_{\beta},\mathbb{R})$ is Baire  for all $\beta\in
A$. Then $\prod\limits_{\beta\in A}
B_{\alpha}(X_{\beta},\mathbb{R})$ is Baire.
\end{lemma}

\begin{proof} Since Lemma holds for the cases $\alpha=0$ (see Corollary
3 in \cite{pyt1}) and  $\alpha=1$ (\cite{Os},Corollary 3.8), we
consider $\alpha\geq 2$.

The space $\prod\limits_{\beta\in A}
B_{\alpha}(X_{\beta},\mathbb{R})=B_{\alpha}(\bigoplus\limits_{\beta\in
A} X_{\beta},\mathbb{R})$ (see Corollary 3.8 in \cite{Os}), hence,
$\prod\limits_{\beta\in A} B_{\alpha}(X_{\beta},\mathbb{R})$ is
Baire (see  Corollary 4.2 in \cite{BG}).

\end{proof}

The following Theorem is the generalization of Theorem 3.15 in
\cite{ospyt1} considering for case $\alpha=0$.

\begin{theorem}\label{th1} Let $X$ be a topological space and $1\leq \alpha\leq \omega_1$. Then the following assertions are equivalent:

$(a)$ $B_{\alpha}(X,\mathbb{R})$ is Baire;

$(b)$ $B_{\alpha}(X,M)$ is Baire for some Fr\'{e}chet space $M$,
$dim M>0$;

$(c)$ $B_{\alpha}(X,L)$ is Baire for any Fr\'{e}chet space $L$.

\end{theorem}

\begin{proof} $(b)\Rightarrow(a)$.
 Let $\varphi$ be a continuous liner functional on $M$,
 $\varphi\neq 0$. By \cite{sh} (see Chapter III, Corollary 1), $\varphi:M\rightarrow \mathbb{R}$ is
 a surjective continuous open mapping. By Lemma \ref{lem6}, $B_{\alpha}(X)$
 is Baire.

$(a)\Rightarrow(c)$.  It is well-known that a space $X$ is Baire
if and only if every open subspace of $X$ is a non-meager.

Assume that $B_{\alpha}(X,L)$ is non-Baire for some Fr\'{e}chet
space $L$.  Then, there is a basis open set $W=\bigcap
\{[x_i,V_i]: i\leq k_0\}=\bigcup\limits_{m=1}^{\infty} F_m$ in
$B_{\alpha}(X,L)$ where $F_m$ is a nowhere dense set and we can
assume that $F_m\subseteq F_{m+1}$ for all $m\in\mathbb{N}$.

Let $\Delta_0=\{x_i:i\leq k_0\}$, $\mu_0=\{V_i:i\leq k_0\}$.
Construct, by induction, finite sets $\Delta_n=\{x_{n,i}: i\leq
k_n\}\subseteq X$, $\Delta_n\cap \Delta_0=\emptyset$,
$n\in\mathbb{N}$, finite families $\mu_n$, $\mu_n\subseteq
\mu_{n+1}$, of open sets in $L$, open sets $V_{n,i}\in \mu_n$ in
$L$, $i\leq k_n$, $n\geq 2$, separable linear subspaces $L_n$ of
$L$, $L_n\subseteq L_{n+1}$, $n\in\mathbb{N}$, families
$\mathcal{B}_n=\{O_{n,m}: m\in\mathbb{N}\}$ of open sets in $L$
for $n\in\mathbb{N}$ such that

(1) for every basic set $\bigcap \{[x,V(x)]: x\in
\bigcup\limits_{i=0}^n \Delta_i\}$, $n>1$, where $V(x)\in\mu_n$,
$V(x_i)\subseteq V_i$, $i\leq k_0$ there are $V'(x)\in\mu_{n+1}$,
$V'(x)\subseteq V(x)$, $x\in \bigcup\limits_{i=0}^n \Delta_i$ such
that $\bigcap \{[x,V'(x)]: x\in \bigcup\limits_{i=0}^n
\Delta_i\}\cap \bigcap\{[x_{n+1,i},V_{n+1,i}]: i\leq k_{n+1}\}\cap
F_n=\emptyset$,

(2) for every $V\in \mu_n$ there is $V'\in \mu_{n+1}$ such that
$\overline{V'}\subseteq V$, $n\geq 0$,

(3) $\mathcal{B}_n=\{O_{n,m}:m\in \mathbb{N}\}$ such that
$\{O_{n,m}\cap L_n: m\in \mathbb{N}\}$ is a base in $L_n$ and
$diam O_{n,m}\rightarrow 0$ ($m\rightarrow \infty$), $n\in
\mathbb{N}$,

(4) $\mu_n\cup\{O_{k,m}: k,m\leq n+1\}\subseteq \mu_{n+1}$, $n\in
\mathbb{N}$.

Let $L^*=\overline{\bigcup\limits_{n=1}^{\infty} L_n}$. Then $L^*$
is a separable Fr\'{e}chet space. By (3) and (4), $V\cap L^*\neq
\emptyset$ for all $V\in \bigcup\limits_{n=0}^{\infty} \mu_n$ and
$\{V\cap L^*: V\in \bigcup\limits_{n=0}^{\infty} \mu_n\}$ is a
$\pi$-base in $L^*$. By \cite{Tur}, $L^*$ is homeomorphic to
$\mathbb{R}^{\gamma}$, $\gamma\leq \aleph_0$.

Since $C_p(X, \mathbb{R}^{\gamma})=\prod\limits_{i<\gamma}
C_p(X,\mathbb{R}_i)$ where $\mathbb{R}_i=\mathbb{R}$ for each
$i<\gamma$ (Proposition 2.6.10 in \cite{Eng}), we get, by
induction of $\beta$, that $B_{\beta}(X,
\mathbb{R}^{\gamma})=\prod\limits_{i<\gamma}
B_{\beta}(X,\mathbb{R}_i)$ for each $\beta\leq \alpha$. Hence,
$B_{\alpha}(X, L^*)\simeq B_{\alpha}(X,
\mathbb{R}^{\gamma})=\prod\limits_{i<\gamma}
B_{\alpha}(X,\mathbb{R}_i)$ and, by Lemma \ref{lem1},
$B_{\alpha}(X,L^*)$ is Baire.

 Let
$M_{p+1}=\bigcup \{\bigcap\{[(x,V(x)\cap L^*]: x\in
\bigcup\limits_{i=0}^{n+1} \Delta_i$ where $V(x_i)\subseteq V_i$,
$i\leq k_0$, $n\geq p+1$, $V(x_{n+1,i})\subseteq V_{n+1,i}$,
$i\leq k_{n+1}$, $V(x)\in \bigcup\limits_{i=0}^{n+1} \mu_i$ and
$\bigcap \{[x,V(x)]: x\in \bigcup\limits_{i=0}^{n+1}
\Delta_i\}\cap F_n=\emptyset\}$, $p\in \mathbb{N}$. Then $M_{n+1}$
is a non-empty open subset of $B_{\alpha}(X,L^*)$ for $n\geq 0$.

We claim that $M_{p+1}$ is dense in $P^*=\bigcap \{[x_i, V_i\cap
L^*]: i\leq k_0\}\subseteq B_{\alpha}(X,L^*)$. Let $\varphi\in
P^*$ and let $O(\varphi)$ be a base neighborhood of $\varphi$ in
$B_{\alpha}(X,L^*)$. We can assume that $O(\varphi)=\bigcap
\{[x,W(x)]: x\in \bigcup\limits_{i=0}^{m}
\Delta_i\}\cap\bigcap\{[y,W(y)]:y\in T\subseteq X\setminus
\bigcup\limits_{i=0}^{\infty} \Delta_i\}$ where $W(x_i)\subseteq
V_i\cap L^*$, $i\leq k_0$ and $W(x)$ ( $x\in
\bigcup\limits_{i=0}^{m} \Delta_i$)and $W(y)$ ($y\in T$) are
non-empty open sets in $L^*$ , $T$ is finite and $m\in\mathbb{N}$.

Since $\{V\cap L^*: V\in\bigcup\limits_{n=0}^{\infty} \mu_n\}$ is
a $\pi$-base of $L^*$ there are $x\in \bigcup\limits_{i=0}^{m}
\Delta_i$ and $V(x)\in \bigcup\limits_{n=0}^{\infty} \mu_n$  such
that $V(x)\cap L^*\subseteq W(x)$ for $x\in
\bigcup\limits_{i=0}^{m} \Delta_i$. Then, there is
$k\in\mathbb{N}$ such that $V(x)\in \mu_l$ for
$x\in\bigcup\limits_{i=0}^{m} \Delta_i$, $k\geq m$, $k\geq p+1$.
By (3), there are sets $V'(x)\in \mu_{l+1}$,
$x\in\bigcup\limits_{i=0}^{l+1} \Delta_i$ such that
$V'(x)\subseteq V(x)$, $x\in \bigcup\limits_{i=0}^{m} \Delta_i$,
$V'(x_{l+1,i})\subseteq V_{l+1,i}$, $i\leq k_{l+1}$ and $\bigcap
\{[x,V'(x)]: x\in\bigcup\limits_{i=0}^{k+1} \Delta_i\}\cap
F_k=\emptyset$.

Choose $g\in \bigcap \{[x,V'(x)\cap L^*]:
x\in\bigcup\limits_{i=0}^{k+1} \Delta_i\}\cap \bigcap \{[y,W(y)]:
y\in T\}$. Then $g\in O(\varphi)\cap P^*\cap M_{p+1}$. Hence,
$M_{p+1}$ is a dense open set in $P^*$. Since $B_{\alpha}(X,L^*)$
is Baire, $\bigcap\limits_{p=0}^{\infty} M_{p+1}\neq \emptyset$.
Let $g\in \bigcap\limits_{p=0}^{\infty} M_{p+1}$. So we proved
that $g\not\in \bigcup\limits_{m=1}^{\infty} F_m=W$ which is a
contradiction.

$(c)\Rightarrow(b)$. It is trivial.

\end{proof}

By Theorem \ref{th1} and Theorem 3.15 in \cite{ospyt1} and
Toru\'{n}czyk results in \cite{Tur}, we get the following
corollary.

\begin{corollary} Let $X$ be a topological space and $0\leq \alpha\leq \omega_1$. Then the following assertions are equivalent:

$(a)$ $B_{\alpha}(X,\mathbb{R})$ is Baire;

$(b)$ $B_{\alpha}(X,Y)$ is Baire for any Banach space $Y$.

\end{corollary}

A collection $\mathcal{G}$ of subsets of $X$ is {\it discrete} if
each point of $X$ has a neighborhood meeting at most one element
of $\mathcal{G}$, and is {\it strongly discrete} if for each $G\in
\mathcal{G}$ there is an open superset $U_G$ of $G$ such that
$\{U_G: G\in\mathcal{G}\}$ is discrete.

By results in \cite{pyt1, vD,ospyt1,  tk}, we get the following
corollary for the case $\alpha=0$.

\begin{corollary}\label{cor9} Let $X$ be a topological space. Then the following assertions are equivalent:

$(a)$ $C_p(X,\mathbb{R})$ is Baire;

$(b)$ $C_p(X,M)$ is Baire for some Fr\'{e}chet space $M$, $dim
M>0$;

$(c)$ $C_p(X,Y)$ is Baire for any Fr\'{e}chet space $Y$;

$(d)$ every pairwise disjoint sequence of non-empty finite subsets
of $X$ has a strongly discrete subsequence.
\end{corollary}

A set of the form $f^{-1}(\{0\})$ for some real-valued continuous
function $f$ on $X$ is called a {\it zero-set} of $X$. A subset
$O\subseteq X$  is called  a cozero-set (functionally open) of $X$
if $X\setminus O$ is a zero-set of $X$. Countable intersection of
cozero sets by $Coz_{\delta}$ (or $Coz_{\delta}(X)$). A
$Coz_{\delta}$-subset of $X$ containing $x$ is called a {\it
$Coz_{\delta}$ neighborhood} of $x$.

 A set $A\subseteq X$ is called {\it strongly
$Coz_{\delta}$-disjoint}, if there is a pairwise disjoint
collection $\{F_a: F_a$ is a $Coz_{\delta}$ neighborhood of $a$,
$a\in A\}$  such that $\{F_a: a\in A\}$ is a {\it completely
$Coz_{\delta}$-additive system}, i.e. $\bigcup\limits_{b\in B}
F_b\in Coz_{\delta}$ for each $B\subseteq A$.

A disjoint sequence $\{\Delta_n: n\in \mathbb{N}\}$ of (finite)
sets is  {\it strongly $Coz_{\delta}$-disjoint} if the set
$\bigcup\{\Delta_n: n\in \mathbb{N}\}$ is strongly
$Coz_{\delta}$-disjoint.

\medskip

In \cite{Os}, we have obtained a characterization when a function
space $B_1(X,\mathbb{R})$ is Baire for a topological space $X$. By
results in \cite{Os} and Theorem \ref{th1}, we get the following
result for the case $\alpha=1$.

\begin{corollary}\label{cor10} Let $X$ be a topological space. Then the following assertions are equivalent:

$(a)$ $B_{1}(X,\mathbb{R})$ is Baire;

$(b)$ $B_{1}(X,M)$ is Baire for some Fr\'{e}chet space $M$, $dim
M>0$;

$(c)$ $B_{1}(X,Y)$ is Baire for any Fr\'{e}chet space $Y$;

$(d)$ every pairwise disjoint sequence of non-empty finite subsets
of $X$ has a strongly $Coz_{\delta}$-disjoint subsequence.
\end{corollary}

By Theorem 4.1 in \cite{BG} and Theorem \ref{th1}, we get the
following result for the cases $\alpha\geq 2$.

\begin{corollary} For any topological space $X$, any Fr\'{e}chet space
$Y$ and every ordinal $2\leq \alpha\leq \omega_1$, the function
space $B_{\alpha}(X,Y)$ is Baire.
\end{corollary}

In particular, for separating examples, by results in
\cite{pyt1,BG, Os}, we get that for any Fr\'{e}chet space $Y$
holds

1. for every $\alpha\geq 2$ the space $B_{\alpha}(\mathbb{R},Y)$
is Baire, but $B_1(\mathbb{R},Y)$ is not.

2. the space $B_1(\mathbb{Q},Y)$ is Baire, but $C_p(\mathbb{Q},Y)$
is not.

\section{Completeness  and compactness properties of $B_{\alpha}(X,Y)$}

A product of two Baire metric spaces need not be Baire \cite{Ox1}.
This result prompted for a search of subclasses of the class of
Baire spaces which are closed under taking arbitrary Cartesian
products. The first such class was introduced by J.C. Oxtoby
\cite{Ox1} under the name {\it pseudocomplete} spaces.

\begin{definition}

 (1) A family $\mathcal{B}$ of sets in a topological space $X$ is
called {\it $\pi$-base} (respectively, {\it $\pi$-pseudobase}) if
every element of $\mathcal{B}$ is open (respectively, has a
nonempty interior) and every nonempty open set in $X$ contains an
element of $\mathcal{B}$.

(2) A space is {\it Oxtoby complete} (respectively, {\it Todd
complete}, {\it Telgarsky complete}) if there is a sequence
$\{\mathcal{B}_n: n\in \mathbb{N}\}$ of $\pi$-bases,
(respectively, $\pi$-pseudobases, bases) in $X$ such that for any
sequence $\{U_n: n\in \mathbb{N}\}$ where $U_n\in \mathcal{B}_n$
and $cl_X U_{n+1}\subseteq Int U_n$ for all $n$, then
$\bigcap\limits_{n\in \mathbb{N}} U_n\neq \emptyset$.

(3) A space is {\it S\'{a}nchez-Okunev complete}  if there is a
sequence $\{\mathcal{B}_n: n\in \mathbb{N}\}$ of $\pi$-bases
consists of zero-sets in $X$ such that for any sequence $\{U_n:
n\in \mathbb{N}\}$ where $U_n\in \mathcal{B}_n$ and
$U_{n+1}\subseteq Int U_n$ for all $n$, then $\bigcap\limits_{n\in
\mathbb{N}} U_n\neq \emptyset$.
\end{definition}

It is a well-known that any Todd complete space is Baire and any
$\check{C}$ech-complete space is Oxtoby complete. Note that if $X$
has a dense Oxtoby complete subspace (in particular, if $X$ has a
dense $\check{C}$ech-complete subspace) then $X$ is Oxtoby
complete (p. 47 in \protect\cite{Tk}).

\begin{lemma}(Propositions 6.3 and 6.4 in \cite{DARHTM})\label{lem42}
Suppose that $X$ is Oxtoby complete (Todd complete) and
$f:X\rightarrow Y$ is a continuous, open and onto function. If $Y$
is metrizable, then $Y$ contains a dense, complete metrizable,
zero-dimensional subspace. In particular, $Y$ is Oxtoby complete
(Todd complete).
\end{lemma}

\begin{lemma}(\cite{Os})\label{lem43} Let $Y$ be a topological vector space and
$L\subseteq Y$ a dense $\check{C}$ech-complete subspace. Then the
linear span of $L$ equal to $Y$.
\end{lemma}

A subset $A$ of a topological space $X$ is called {\it
$G_{\delta}$-dense} in $X$ if $A$ has nonempty intersection with
any nonempty $G_{\delta}$-set in $X$.

A family $\mathcal{F}\subseteq Y^X$ of functions from a set $X$ to
a set $Y$ is called {\it $\omega$-full} in $Y^X$ if each function
$f:Z\rightarrow Y$ defined on a countable subset $Z\subseteq X$
has an extension $f\in \mathcal{F}$.

Both Baire and meager spaces have game characterizations due to
Oxtoby. Oxtoby considered two infinite games $G_E(X)$ and $G_N(X)$
played by two players $E$ and $N$ on a topological space $X$
\cite{Ox}.

\begin{theorem}\label{th40}(Oxtoby). A topological space $X$ is

(1) meager iff the player $E$ has a winning strategy in the game
$G_N(X)$;

(2) Baire iff the player $E$ has no winning strategy in the game
$G_E(X)$.

\end{theorem}

A topological space $X$ is defined to be {\it Choquet} if the
player $N$ has a winning strategy in the game $G_E(X)$. Choquet
spaces were introduced by White \cite{Wh} who called them {\it
weakly $\alpha$-favorable} spaces.

A topological space $X$ is called {\it strong Choquet} if the
player $N$ has a winning strategy in a modification $G^s_E(X)$ of
the game $G_E(X)$. More information on (strong) Choquet spaces can
be found in (\cite{Ke}, 8.CD).

A topological space $X$ is called {\it countably base-compact} if
it has a base $\mathcal{B}$ of the topology such that for any
countable centered subfamily $\mathcal{S}\subseteq \mathcal{B}$
the intersection $\bigcap\limits_{S\in \mathcal{S}} \overline{S}$
is not empty \cite{AL}.

The following result was proved in the paper (\cite{BG}, Theorem
2.20) for each dense subgroup of $\mathbb{R}^{\kappa}$ where
$\kappa$ is any cardinal number.

\begin{theorem}\label{th30} For a topological space $X$, any Fr\'{e}chet space $Y$ and each dense
subgroup $G$ of $Y^X$, the following assertions are equivalent:

$(1)$ $G$ is $\omega$-full in $Y^X$;

$(2)$ $G$ is $G_{\delta}$-dense in $Y^X$;

$(3)$ $G$ is countably base-compact;

$(4)$ $G$ is strong Choquet;

$(5)$ $G$ is Choquet.
\end{theorem}

\begin{proof} $(1)\Leftrightarrow(2)$. By Proposition 2.15 in
\cite{BG}.

$(2)\Rightarrow(3)$.  A topological space is {\it Moscow} if the
closure of every open subset is a union of $G_{\delta}$-sets
(\cite{at}, \S 6.1). By (\cite{at}, p.347) the Tychonoff product
of first-countable space is Moscow and every dense subspace of a
Moscow space is Moscow.

Since $Y$ is a first-countable space, the space $Y^X$ is Moscow.
Since $Y$ is complete metrizable, it is countably base-compact. By
(\cite{AL}, Theorem 3.b), the countably base-compactness is
preserved by Tychonoff products. It follows that $Y^X$ is
countably base-compact.

By (\cite{BG}, Proposition 2.18) every $G_{\delta}$-dense subspace
$X$ of a countably base-compact Moscow space $Y$ is countably
base-compact.

Since $G$ is $G_{\delta}$-dense subspace of a countably
base-compact Moscow space $Y^X$, $G$ is countably base-compact.

$(3)\Rightarrow(4)\Rightarrow(5)$. These implications can be
easily derived from the definitions (see also in \cite{BG}).

$(5)\Rightarrow(2)$. By Theorem 3 of \cite{BH}, a topological
group $X$ is Choquet if and only if the Raikov completion $\varrho
X$ of $X$ is Choquet and $X$ is $G_{\delta}$-dense in $\varrho X$.

Since $Y$ is complete metrizable,  $Y^X=\varrho Y^X$ and $\varrho
G=Y^X$. It follows that $G$ is $G_{\delta}$-dense in $Y^X$.

\end{proof}

\begin{definition}
A family $\mathcal{F}$ of non-empty subsets of a set $X$ will be
called {\it countably compact} if every decreasing sequence
$\{F_n: n\in \mathbb{N}\}\subseteq \mathbb{F}$ has non-empty
intersection (Definition 3.1 in \cite{DAS}).
\end{definition}

\begin{definition}(Definition 3.4 in \cite{DAS}) A topological
space $X$ will be called:

(1) {\it Todd countably compact} if $X$ has a countably compact
pseudobase;

(2) {\it Oxtoby countably compact} if $X$ has a countably compact
pseudobase consisting of open sets;

(3) {\it S\'{a}nchez-Okunev countably compact} if $X$ has a
countably compact pseudobase which consists of zero-sets of $X$;

(4) {\it Telg\'{a}rsky countably compact} if $X$ has a countably
compact base.

\end{definition}

Highlights here include the coincidence of various completeness
and compactness properties (see definitions in \cite{BG,DAS}) in
$B_{\alpha}(X,Y)$ for an arbitrary Fr\'{e}chet space $Y$ and
characterization of these completeness and compactness properties
in terms of internal properties of the space $X$.

Let $\mathcal{P}$ be one of the following properties: { \it
Telg\'{a}rsky complete, Todd complete, Oxtoby complete,
S\'{a}nchez-Okunev complete, strongly Telg\'{a}rsky complete,
strongly Oxtoby complete, strongly Todd complete, strongly
S\'{a}nchez-Okunev complete, Telg\'{a}rsky countably compact,
Oxtoby countably compact, Todd countably compact and
S\'{a}nchez-Okunev countably compact, countably base-compact,
strong Choquet, Choquet}.

\medskip

The Oxtoby Theorem \ref{th40} and the aforementioned results
(\cite{DAS}, see Diagram 2, \cite{DARHTM}) imply the following
implications of properties.

\medskip

\begin{figure}[h]%
\centering
\includegraphics[width=0.9\textwidth]{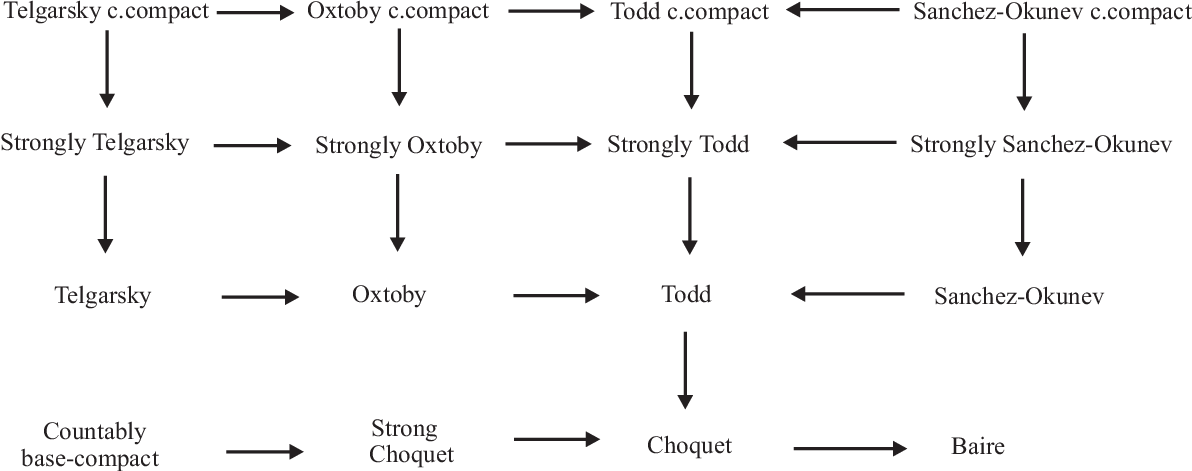}
%\caption{Diagram~1.}\label{fig1}
\end{figure}

\medskip

 By Theorems 4.1 and 4.2 in \cite{DAS}, we have a following
proposition.

\begin{proposition} Any \v{C}ech-complete metric group has property
$\mathcal{P}$. In particular, any Fr\'{e}chet space has property
$\mathcal{P}$.
\end{proposition}

\newpage

\begin{theorem}\label{th11} Let $X$ be a topological space and let $Y$ be any Fr\'{e}chet
space.  Then for every $0\leq \alpha\leq \omega_1$ the following
assertions are equivalent:

$(1)$ $B_{\alpha}(X,Y)$ is $G_{\delta}$-dense in $Y^X$;

$(2)$ $B_{\alpha}(X,Y)$ is $\omega$-full in $Y^X$;

$(3)$ $B_{\alpha}(X,Y)$ has the property $\mathcal{P}$;

$(4)$ $B_{\alpha}(X,\mathbb{R})$ is $\omega$-full in
$\mathbb{R}^X$;

$(5)$ $B_{\alpha}(X,Y)$ is countably base-compact;

$(6)$ $B_{\alpha}(X,Y)$ is strong Choquet;

$(7)$ $B_{\alpha}(X,Y)$ is Choquet;

%$(8)$ $B_{\alpha}(X,Y)$ is Telg\'{a}rsky complete;

%$(9)$ $B_{\alpha}(X,Y)$ is Oxtoby complete;

%$(10)$ $B_{\alpha}(X,Y)$ is S\'{a}nchez-Okunev complete;

%$(11)$ $B_{\alpha}(X,Y)$ is strongly Telg\'{a}rsky complete;

%$(12)$ $B_{\alpha}(X,Y)$ is strongly Oxtoby complete;

%$(13)$ $B_{\alpha}(X,Y)$ is strongly Todd complete;

%$(14)$ $B_{\alpha}(X,Y)$ is strongly S\'{a}nchez-Okunev complete;

%$(15)$ $B_{\alpha}(X,Y)$ is Telg\'{a}rsky countably compact;

%$(16)$ $B_{\alpha}(X,Y)$ is Oxtoby countably compact;

%$(17)$ $B_{\alpha}(X,Y)$ is  Todd countably compact;

%$(18)$ $B_{\alpha}(X,Y)$ is S\'{a}nchez-Okunev countably compact;

$(8)$ $B_{\alpha}(X,Y)$ is $C$-embedded in $Y^X$.

%$(19)$ $B_{\alpha}(X,Y)$ is $C$-embedded in $Y^X$;

%$(9)$ $B_{\alpha}(X,\mathbb{R})$ is $C$-embedded in
%$\mathbb{R}^X$.

%$(20)$ $B_{\alpha}(X,\mathbb{R})$ is $C$-embedded in
%$\mathbb{R}^X$.

\end{theorem}

\begin{proof} By Theorem \ref{th30}, the implications $(1)$, $(2)$, $(5)$, $(6)$ and $(7)$ are
equivalents.

$(2)\Rightarrow (3)$. By Corollary 10.8
in \cite{DAS}.

%$(3)\Rightarrow (7)$. By Theorem 6.10 in \cite{DARHTM}.

$(1)\Leftrightarrow(8)$. By Corollary 6.1.10 in \cite{at}.

$(3)\Rightarrow(2)$. Let $A$ be a countable subspace of $X$.
Consider the projection map $\psi: Y^X\rightarrow Y^A$, i.e.,
$\psi(f)=f\upharpoonright A$ for every $f\in Y^X$. The mapping
$\psi$ is continuous and open. Note that $Y^A$ is a complete
metrizable space.  By Lemma \ref{lem42}, there is $M_1\subseteq
\psi(B_{\alpha}(X,Y))$ such that $M_1$ is a dense $G_{\delta}$-set
in $Y^A$. Then, by Lemma \ref{lem43}, a linear span $\langle M_1
\rangle$ of $M_1$ equal to $Y^A$, but
$\psi(B_{\alpha}(X,Y))\supseteq \langle M_1 \rangle=Y^A$, hence,
$\psi(B_{\alpha}(X,Y))=Y^A$.

$(2)\Rightarrow(4)$. Consider $\mathbb{R}$ as a subspace of $Y$.
Let $A$ be a countable subspace of $X$ and $f: A\rightarrow
\mathbb{R}\subset Y$. Since $B_{\alpha}(X,Y)$ is $\omega$-full in
$Y^X$ there is $\widetilde{f}\in B_{\alpha}(X,Y)$ such that
$\widetilde{f}\upharpoonright A=f$. Let $r: Y\rightarrow
\mathbb{R}$ be a retraction mapping, i.e., $r$ is continuous and
$r(x)=x$ for every $x\in \mathbb{R}$. Then
$F=r\circ\widetilde{f}:X\rightarrow \mathbb{R}$ is an extension
$f$ on $X$.

$(4)\Rightarrow(2)$. Let $A$ be a countable subspace of $X$ and
$f: A\rightarrow Y$. Consider a linear span $L=\langle f(A)
\rangle$ of a countable subspace $f(A)$ of $Y$. Then
$\overline{L}$ is a separable Fr\'{e}chet space. By Toru\'{n}czyk
results in \cite{Tur}, $\overline{L}$ is homeomorphic to
$\mathbb{R}^\kappa$ for some  $\kappa\leq \aleph_0$. Note that if
$B_{\alpha}(X,\mathbb{R})$ is $\omega$-full in $\mathbb{R}^X$ then
$B_{\alpha}(X,\mathbb{R}^\kappa)$ is $\omega$-full in
$\mathbb{R}^{\kappa}$ for any cardinal $\kappa$ (Theorem 2.20 in
\cite{BG}). Thus there is $\widetilde{f}:X\rightarrow
\mathbb{R}^{\kappa}\subset Y$ such that $\widetilde{f}\in
B_{\alpha}(X,Y)$ and $\widetilde{f}\upharpoonright A=f$.

%Remain implications holds by definitions of completeness and
%compactness properties (see Diagram 1) and results in
%(\cite{DAS,DARHTM}).

\end{proof}

\begin{definition} Let $X$ and $Y$ be a topological spaces.

(1) A subspace $M$ of $X$ is {\it $C_Y$-embedded} in $X$ if every
continuous function $f:M\rightarrow Y$ has a continuous extension
over $X$.

(2) The space $X$ is {\it $b_Y$-discrete} if every countable subspace of $X$ is discrete
and $C_Y$-embedded in $X$.
\end{definition}

By Theorem \ref{th11} for $\alpha=0$ and results in
\cite{pyt1,tk,DAS}, we get the following corollary.

\begin{corollary} Let $X$ be a space and let $Y$ be any Fr\'{e}chet
space. Then  the following assertions are equivalent:

$(1)$ $C_{p}(X,Y)$ is $G_{\delta}$-dense in $Y^X$;

$(2)$ $C_p(X,\mathbb{R})$ satisfies $P\in \mathcal{P}$;

$(3)$ $C_p(X,Y)$ satisfies $P\in \mathcal{P}$;

$(4)$ $X$ is $b_\mathbb{R}$-discrete;

$(5)$ $X$ is $b_Y$-discrete;

$(6)$ Every countable subset of $X$ is strongly discrete.

\end{corollary}

\begin{definition} Let $X$ and $Y$ be a topological spaces

(1) A subspace $M$ of $X$ is $B^1_Y$-embedded in $X$ if every
Baire-one function $f:M\rightarrow Y$ has a Baire-one extension
over $X$.

(2) The space $X$ is {\it $b^1_Y$-discrete} if every countable
subspace of $X$ is $B^1_Y$-embedded in $X$.

\end{definition}

By results in \cite{Os} and Theorem \ref{th11} for $\alpha=1$, we
have the following theorem.

\begin{corollary} Let $X$ be a space and let $Y$ be any Fr\'{e}chet
space. Then  the following assertions are equivalent:

$(1)$ $B_1(X,Y)$ is $G_{\delta}$-dense in $Y^X$;

$(2)$ $B_1(X,\mathbb{R})$ satisfies $P\in \mathcal{P}$;

$(3)$ $B_1(X,Y)$ satisfies $P\in \mathcal{P}$;

$(4)$ $X$ is $b^1_\mathbb{R}$-discrete;

$(5)$ Every countable subset of $X$ is strongly
$Coz_{\delta}$-disjoint.
\end{corollary}

Recall that a topological space $X$ is called a {\it
$\lambda$-space} if every countable subset is of type $G_{\delta}$
in $X$.

\begin{corollary} Let $X$ be a space of countable pseudocharacter and let $Y$ be any Fr\'{e}chet
space. Then  the following assertions are equivalent:

$(1)$ $B_1(X,Y)$ is $G_{\delta}$-dense in $Y^X$;

$(2)$ $B_1(X,Y)$ satisfies $P\in \mathcal{P}$;

$(3)$ $X$ is $\lambda$-space.
\end{corollary}

Since $B_{\alpha}(X,\mathbb{R})$ is Choquet for every $\alpha\geq
2$ (Corollary 4.2 in \cite{BG}),  Theorem \ref{th11} immediately
imply the following result.

\begin{corollary} For any topological space $X$, any Fr\'{e}chet space
$Y$ and every ordinal $2\leq \alpha\leq \omega_1$, the function
space $B_{\alpha}(X,Y)$ satisfies $P\in \mathcal{P}$.
\end{corollary}

%%=============================================%%
%% For presentation purpose, we have included  %%
%% \bigskip command. please ignore this.       %%
%%=============================================%%

%% For presentation purpose, we have included  %%
%% \bigskip command. please ignore this.       %%
%%=============================================%%
%\bigskip
%\begin{verbatim}
%\begin{figure}[<placement-specifier>]
%\centering
%\includegraphics{<eps-file>}
%\caption{<figure-caption>}\label{<figure-label>}
%\end{figure}
%\end{verbatim}
%\bigskip
%%=============================================%%
%% For presentation purpose, we have included  %%
%% \bigskip command. please ignore this.       %%
%%=============================================%%

\medskip

{\bf Author contributions} All the results of the paper were
obtained by the author himself.

\medskip

{\bf Data Availability Statement}

Data sharing is not applicable to this article as no new data were
created or analyzed in this study.

\bibliographystyle{model1a-num-names}
\bibliography{<your-bib-database>}

\begin{thebibliography}{24}

\bibitem{AL}
J. Aarts, D. Lutzer, Completeness properties designed for
recognizing Baire spaces, Dissert. Math. 116 (1974), 1--43.


\bibitem{at}
A.V. Arkhangel'skii, M.G. Tkachenko, Topological group and related
structures, Atlantis Press/World Scientific, Amsterdam-Raris,
2008.

\bibitem{BH}
T. Banakh, O. Hryniv, Some Baire category properties of
topological groups, Visnyk Lviv Univ. Ser. Mech.-Mat.  86 (2018)
71--76.

\bibitem{Eng}
R. Engelking, General Topology, Revised and completed edition,
Heldermann Verlag Berlin (1989).

\bibitem{pyt1}
E.G. Pytkeev, Baire property of spaces of continuous functions.
Mathematical Notes Acad. Sci. USSR,  38 (1985), 908--915.
(Translated from Matematicheskie Zametki, 38:5, 726--740).

\bibitem{Kr}
K. Kuratovski,  Topology I, Academic Press, New York 1966.

\bibitem{LM}
D. Lutzer, R. McCoy, Category in function spaces. I, Pacific J.
Math. 90 (1980), no.1, 145--168.



\bibitem{vD}
J.van Mill ed, Eric K. van Douwen, Collected Papers, vol. 1,
North-Holland, Amsterdam (1994).


\bibitem{BG}
T. Banakh, S. Gabriyelyan, Baire category of some Baire type
function spaces, Topology and its Applications, 272 (2020),
107078.

\bibitem{ospyt}
A.V.~Osipov, E.G.~Pytkeev, Baire property of spaces of
$[0,1]$-valued continuous functions, Rev. Real Acad. Cienc.
Exactas Fis. Nat. Ser. A -- Mat. {\bf 117}, 38 (2023).


\bibitem{ospyt1}
A.V.~Osipov, E.G.~Pytkeev, Baire property of some function spaces,
Acta Math. Hungar. {\bf 168}(2022) 246--259.




\bibitem{sh}
H.H.~Schaefer, Topological Vector Spaces, Springer-Verlag New
York, 1971.

\bibitem{Tur}
H.~Toru\'{n}czyk, Characterizing Hilbert space topology, Fund.
Math. {\bf 111} (1981), 247--262.



\bibitem{tk}
V.V. Tkachuk, Characterization of the Baire property in $C_p(X)$
by the properties of the space $X$, Researh papers, Topology-Maps
and Extensions of Topological Spaces (Ustinov)(1985), 21--27.

\bibitem{Os}
A.V. Osipov, Baire property of the space of Baire-one functions. European Journal of Mathematics 11, 8 (2025).



\bibitem{Tk}
V.V. Tkachuk, A $C_p$-Theory Problem Book. Topological and
Function spaces., Springer, 2011.

\bibitem{Ox}
J.C. Oxtoby, The Banach-Mazur game and Banach Category Theorem,
in: Contributions to the theory of games, Vol. III, Annals.
Studies {\bf 39} (1957), 159--163.




\bibitem{Ox1}
J.C. Oxtoby, Cartezian products of Baire spaces, Fund. Math. {\bf
49}:2 (1961), 157--166.

\bibitem{DAS}
A. Dorantes-Aldama, D. Shakhmatov, Completeness and compactness
properties in metric spaces, topological group and function
spaces, Topology and its Applications, {\bf 226}, 2017, 134--164.



\bibitem{Is}
T. Ishii, The Tychonoff functor and related topics, in: Topics in
General Topology, K. Morita and J. Nagata, eds., North- Holland
(1989), 203--243.

\bibitem{Ok}
S. Oka, The Tychonoff functor and product spaces, Proc. Japan
Acad. 54 (1978), 97--100.


\bibitem{Wh}
H.E. White, Jr., Topological spaces that are $\alpha$-favorable
for a player with perfect information. Proc. Amer. Math. Soc. {\bf
50} (1975), 477--482.

\bibitem{Ke}
A. Kechris, Classical descriptive set theory, 156.
Springer-Verlag, New York, 1995.

\bibitem{DARHTM}
A. Dorantes-Aldama, R. Rojas-Hern\'{a}ndez, \'{A}.
Tamariz-Mascar\'{u}a, Weak pseudocompactness on spaces of
continuous functions, Topology and its Applications, {\bf 196}
(2015), 72--91.

\end{thebibliography}

\end{document}